\documentclass{birkjour}

\usepackage{amssymb,amsmath,latexsym,amsthm,amsfonts}
\usepackage{blindtext}
\usepackage{todonotes}
\usepackage{esint,dsfont}

\usepackage{hyperref}

\def\N{\mathbb{N}}
\def\C{\mathbb{C}}
\def\Z{\mathbb{Z}}

\def\F{\mathbb{F}}

\newtheorem{prop}{\bf Proposition}[section]
\newtheorem{thm}[prop]{\bf Theorem}

\newtheorem{rmk}[prop]{\it Remark}

\begin{document}

\title[Weak amenability of free products]{Weak amenability of free products of hyperbolic and amenable groups}

\author[I. Vergara]{Ignacio Vergara}

\address{Saint-Petersburg State University, Leonhard Euler International Mathematical Institute,
14th Line 29B, Vasilyevsky Island, St. Petersburg, 199178, Russia}

\email{ign.vergara.s@gmail.com}

\thanks{This work is supported by the Ministry of Science and Higher Education of the Russian Federation, agreement № 075–15–2019–1619}

%    General info
%
%\date{\today}


\begin{abstract}
We show that, if $G$ is an amenable group and $H$ is a hyperbolic group, then the free product $G\ast H$ is weakly amenable. A key ingredient in the proof is the fact that $G\ast H$ is orbit equivalent to $\Z\ast H$.
\end{abstract}

\maketitle

\section{Introduction}

Weak amenability was introduced by Cowling and Haagerup \cite{CowHaa} as a generalisation of amenability in the context of approximation properties of operator algebras. Before receiving that name, this property was first studied by de Canni\`ere and Haagerup \cite{deCHaa} for discrete subgroups of $\operatorname{SO}_0(n,1)$. Although weak amenability can be defined for general locally compact groups by means of approximate identities of the Fourier algebra, in this paper we will restrict ourselves to countable discrete groups. We refer the reader to \cite{CowHaa} for the general definition.

Let $G$ be a countable group. We say that $G$ is weakly amenable if there exists a sequence of finitely supported functions $\varphi_n:G\to\C$ converging pointwise to $1$, and a constant $C\geq 1$ such that
\begin{align*}
\sup_{n\in\N}\|\varphi_n\|_{B_2(G)}\leq C.
\end{align*}
Here $B_2(G)$ stands for the space of Herz--Schur multipliers on $G$. See \S\ref{Sec_wa} for details. The Cowling--Haagerup constant $\boldsymbol\Lambda(G)$ is the infimum of all $C\geq 1$ such that the condition above holds. If $G$ is not weakly amenable, we set $\boldsymbol\Lambda(G)=\infty$. 

Recall that a group $G$ is amenable if and only if there exists a sequence of finitely supported, positive definite functions $\varphi_n:G\to\C$ converging pointwise to $1$; see e.g. \cite[\S 2.6]{BroOza}. The fact that every positive definite function $\varphi:G\to\C$ satisfies $\|\varphi\|_{B_2(G)}=\varphi(e)$ implies that every amenable group $G$ is weakly amenable with $\boldsymbol\Lambda(G)=1$. The converse is not true since free groups satisfy $\boldsymbol\Lambda(\F_n)=1$ for all $n\geq 1$.

Weak amenability is known to be preserved by taking direct products and subgroups; however, it is not known if it is stable under free products. For a pair of groups $G,H$, the free product $G\ast H$ is the group of words in $G$ and $H$, where the group operation is defined by concatenation. See e.g. \cite[\S 2.3.2]{Loh} for the formal definition. The question of weak amenability of free products remains open in general, but some particular cases are well understood.

Bo\.{z}ejko and Picardello \cite{BozPic} showed that the free product of amenable groups is weakly amenable with Cowling--Haagerup constant 1. The proof of this result exploits in a very clever way the geometry of trees, and the techniques developed in that paper proved to be very fruitful in the study of weak amenability from a geometric point of view.

Generalising Bo\.{z}ejko and Picardello's result, Ricard and Xu \cite{RicXu} proved that if $G$ and $H$ are weakly amenable groups with $\boldsymbol\Lambda(G)=\boldsymbol\Lambda(H)=1$, then $G\ast H$ is weakly amenable with $\boldsymbol\Lambda(G\ast H)=1$. In this case, the approach is quite different, as they prove a more general theorem for reduced free products of $\operatorname{C}^*$-algebras, obtaining the result for groups as a corollary.

In \cite{GuReTe}, Guentner, Reckwerdt and Tessera studied weak amenability for relatively hyperbolic groups. A particular case of their result says that, if $H$ is hyperbolic and $G$ has polynomial growth, then $G\ast H$ is weakly amenable. This result builds on Ozawa's work \cite{Oza}, where he shows that all hyperbolic groups are weakly amenable.

The aim of this paper is to prove the following.

\begin{prop}\label{Prop}
Let $G$ be an amenable group and let $H$ be a hyperbolic group. Then the free product $G\ast H$ is weakly amenable.
\end{prop}

The proof relies on the notion of orbit equivalence of group actions, which we review in \S\ref{Sec_OE}.

\section{Weak amenability}\label{Sec_wa}

We will give now the precise definition of weak amenability. For this, we need to talk first about Schur multipliers. For more details, we refer the reader to \cite[Appendix D]{BroOza}.

Let $X$ be a set and let $\ell_2(X)$ denote the Hilbert space of complex-valued, square-summable functions on $X$. We denote by $\delta_x$ ($x\in X$) the elements of the canonical orthonormal basis of $\ell_2(X)$. Let $T:\ell_2(X)\to\ell_2(X)$ be a bounded linear operator. We define its matrix coefficients $(T_{x,y})$ by
\begin{align*}
T_{x,y}=\langle T\delta_y,\delta_x\rangle,\quad\forall x,y\in X.
\end{align*}
Observe that an operator $T$ is completely determined by its matrix coefficients $(T_{x,y})_{x,y\in X}$. We say that a function $\psi:X\times X\to\C$ is a Schur multiplier on $X$ if the map
\begin{align}\label{def_M_psi}
M_\psi : (T_{x,y}) \mapsto (\psi(x,y)T_{x,y})
\end{align}
is well defined in the algebra of bounded operators $\mathcal{B}(\ell_2(X))$. In this case, $M_\psi$ is automatically continuous.

Now let $G$ be a countable group. We say that $\varphi:G\to\C$ is a Herz--Schur multiplier on $G$ if the function $\psi:G\times G\to\C$ given by
\begin{align*}
\psi(t,s)=\varphi(s^{-1}t),\quad\forall s,t\in G,
\end{align*}
is a Schur multiplier on $G$. We denote the space of Herz--Schur multipliers by $B_2(G)$ and we endow it with the norm
\begin{align*}
\|\varphi\|_{B_2(G)}=\|M_\psi\|,
\end{align*}
where $M_\psi$ is the bounded linear map on $\mathcal{B}(\ell_2(G))$ defined as in \eqref{def_M_psi}.

We say that $G$ is weakly amenable if there exists a sequence of finitely supported Herz--Schur multipliers $\varphi_n:G\to\C$ converging pointwise to $1$, and a constant $C\geq 1$ such that
\begin{align*}
\sup_{n\in\N}\|\varphi_n\|_{B_2(G)}\leq C.
\end{align*}
We define the Cowling--Haagerup constant $\boldsymbol\Lambda(G)$ as the infimum of all $C\geq 1$ such that the condition above holds.

Recall that a finitely generated group is hyperbolic if its Cayley graph is hyperbolic. For more details, we refer the reader to \cite[\S 7]{Loh}. The following result is essential for our purposes.

\begin{thm}[Ozawa \cite{Oza}]\label{Thm_Oza}
Let $H$ be a hyperbolic group. Then $H$ is weakly amenable.
\end{thm}

Observe that this theorem does not provide a bound on the Cowling--Haagerup constant of $H$. And, in fact, there is no such bound. More precisely, it was proved in \cite{CowHaa} that every cocompact lattice $\Gamma$ in $\operatorname{Sp}(n,1)$ satisfies
\begin{align*}
\boldsymbol\Lambda(\Gamma)=2n-1,
\end{align*}
and such lattices are hyperbolic for all $n$.

Cowling and Zimmer \cite{CowZim} exploited this fact in order to show that two lattices $\Gamma<\operatorname{Sp}(n,1)$ and $\Lambda<\operatorname{Sp}(m,1)$ cannot be orbit equivalent if $n\neq m$.

\section{Orbit equivalence}\label{Sec_OE}

A key concept in the study of measured group theory is the notion of orbit equivalence for group actions. We refer the reader to \cite{Fur} for a detailed treatment of this subject. In this article, we will only focus on its connection with weak amenability, as devised by Cowling and Zimmer \cite{CowZim}.

Let $G$ and $H$ be countable groups, and let $G\curvearrowright(X,\mu)$, $H\curvearrowright(Y,\nu)$ be measure preserving actions on standard non-atomic probability measure spaces. We say that these actions are orbit equivalent if there exists a measure space isomorphism $T:(X,\mu)\to (Y,\nu)$ sending $G$-orbits onto $H$-orbits. We call $T$ an orbit equivalence between $G\curvearrowright(X,\mu)$ and $H\curvearrowright(Y,\nu)$. We will write $G\sim_{OE}H$ if such actions and such an orbit equivalence exist. See \cite[\S 2.2]{Fur} for more details.

\begin{thm}[Cowling--Zimmer \cite{CowZim}]\label{Thm_CZ}
Let $G$ and $H$ be countable groups such that $G\sim_{OE}H$. Then $\boldsymbol\Lambda(G)=\boldsymbol\Lambda(H)$.
\end{thm}

\begin{rmk}
This result was later extended by Jolissaint \cite{Jol} to measure equivalent groups. Recently, Ishan \cite{Ish} extended this even further to the context of von Neumann equivalence of groups, which was introduced in \cite{IsPeRu}.
\end{rmk}

A very important result concerning orbit equivalence of amenable groups was proved by Ornstein and Weiss \cite{OrnWei}.

\begin{thm}[Ornstein--Weiss \cite{OrnWei}]\label{Thm_OW}
Any two ergodic probability measure preserving actions of any two infinite countable amenable groups are orbit equivalent.
\end{thm}

In particular, every infinite countable amenable group $G$ satisfies $G\sim_{OE}\Z$.

\begin{proof}[Proof of Proposition \ref{Prop}]
If $G$ is finite, then $G\ast H$ is hyperbolic, and so it is weakly amenable by Theorem \ref{Thm_Oza}. Assume now that $G$ is infinite. Since $G$ is amenable, by Theorem \ref{Thm_OW}, we have $G\sim_{OE}\Z$. Since the relation $\sim_{OE}$ is preserved by free products (see \cite[\S 3]{Fur}), this implies that
\begin{align*}
G\ast H\sim_{OE}\Z\ast H.
\end{align*}
Hence, by Theorem \ref{Thm_CZ},
\begin{align*}
\boldsymbol\Lambda(G\ast H)=\boldsymbol\Lambda(\Z\ast H).
\end{align*}
On the other hand, since $H$ and $\Z$ are both hyperbolic, so is $\Z\ast H$. Therefore, by Theorem \ref{Thm_Oza},
\begin{align*}
\boldsymbol\Lambda(\Z\ast H) < \infty.
\end{align*}
This shows that
\begin{align*}
\boldsymbol\Lambda(G\ast H) < \infty,
\end{align*}
which finishes the proof.
\end{proof}

\bibliographystyle{plain} 

\bibliography{Bibliography}

\end{document}